\documentclass[microtype]{gtpart}
\usepackage{graphicx}
\usepackage[mathscr]{eucal}
\usepackage{amssymb}
\usepackage{xcolor}
\definecolor{darkred}{RGB}{200, 60, 0}
\definecolor{mildblue}{RGB}{0, 100, 250}
\definecolor{Cerulean}{rgb}{0.0, 0.48, 0.65}
\usepackage{enumerate, cite}
\usepackage[margin=1.25in]{geometry}
\usepackage{hyperref}
\hypersetup{
		colorlinks=true,
		linkcolor=darkred,
		urlcolor=darkred,
		citecolor=mildblue,      
		urlcolor=darkred,
}
\usepackage[nameinlink]{cleveref}

\newcommand{\bn}{\mathbb N}

\newcommand{\vp}{\varphi}

\newcommand{\ssm}{\smallsetminus}

\newcommand{\autg}{\Aut(\mathcal S(\Gamma))}

\DeclareMathOperator{\mcg}{MCG}
\DeclareMathOperator{\pmcg}{PMCG}

\DeclareMathOperator{\Homeo}{Homeo}

\DeclareMathOperator{\Aut}{Aut}

\DeclareMathOperator{\supp}{supp}

\DeclareMathOperator{\HH}H
\DeclareMathOperator{\F}F

\renewcommand{\co}{\colon\thinspace}

\usepackage[normalem]{ulem}

\newtheorem{theorem}{Theorem}[section]
\newtheorem{theorem*}{Theorem}
\newtheorem{proposition}[theorem]{Proposition}
\newtheorem{lemma}[theorem]{Lemma}
\newtheorem{corollary}[theorem]{Corollary}
\newtheorem{corollary*}[theorem*]{Corollary}

\theoremstyle{definition}
\newtheorem{construction}{Construction}

\newtheorem*{Ex*}{Examples}
\newtheorem{remark}[theorem]{Remark}

\numberwithin{equation}{section}

\title{Mapping class groups have a unique Polish group structure}

\author{Tyrone Ghaswala}
\address{Center for Education in Mathematics and Computing, University of Waterloo, Waterloo, ON, Canada N2L3G1}\email{tghaswala@uwaterloo.ca}

\author{Sumun Iyer}
\address{Department of Mathematics \ Carnegie Mellon University, \ Pittsburgh, PA 15217 }\email{sumuni@andrew.cmu.edu}

\author{Robert Alonzo Lyman}
\address{Department of Mathematics and Computer Science, \ Rutgers University \ Newark, NJ 07102}
\email{robbie.lyman@rutgers.edu}

\author{Nicholas G. Vlamis}
\address{Department of Mathematics, \ CUNY Graduate Center \ New York, NY 10016, and \newline Department of Mathematics \ CUNY Queens College \ Flushing, NY 11367}
\email{nvlamis@gc.cuny.edu}

\begin{document}

\begin{abstract}
	We prove that mapping class groups of surfaces and of locally finite connected graphs support a unique Polish group structure.
\end{abstract}

\maketitle

\vspace{-20pt}

\section{Introduction}

A topological space is \emph{Polish} if it is separable and completely metrizable.
Let \(G\) be a group.
We say that \(G\) has a \emph{unique Polish group topology} if there is a unique Polish topology on \(G\) making \(G\) a topological group.
Having a unique Polish group topology can be thought of as a rigidity property for the group, as it means that the algebraic structure of the group determines its possible group topologies in a strong way.

The prototypical example of a group without a unique Polish group topology is the group of real numbers \(\R\) under addition.
With the usual topology of the real line, \(\R\) is a Polish group.
However, \(\R\) is isomorphic---as a group---to \(\R^n\) for any \(n\in \mathbb N\), because both \(\R\) and \(\R^n\) are \(\Q\)-vector spaces of dimension \(2^{\aleph_0}\).
Endowing \(\R\) with the topology inherited via a group isomorphism with each \(\R^n\) produces infinitely many distinct Polish group topologies on \(\R\).
One can even endow \(\R\) with a totally disconnected Polish group topology via an isomorphism with the \(p\)-adic rationals \(\Q_p\) with the \(p\)-adic topology, since \(\Q_p\) is also a \(2^{\aleph_0}\)-dimensional vector space over \(\Q\).

For this paper, a \emph{surface} is a connected 2-manifold without boundary.
Given a surface \( S \), the \emph{mapping class group} of \( S \) is the group of isotopy classes of homeomorphisms \( S \to S \). 
In other words, \( \mcg(S) = \Homeo(S) / \Homeo_0(S) \), where \( \Homeo(S) \) is the full group of homeomorphisms \( S \to S \) and \( \Homeo_0(S) \) is the subgroup consisting of homeomorphisms isotopic to the identity. 
Equipped with the compact-open topology, \( \Homeo(S) \) is a Polish group and \( \Homeo_0(S) \) is closed, implying that the corresponding quotient topology on \( \mcg(S) \) is Polish.
Our first theorem says that this is the unique Polish group structure on \( \mcg(S) \). 

\begin{theorem}\label{thm_1_intro}
    Let \(S\) be a surface. The group \(\textrm{MCG}(S)\) has a unique Polish group topology.
\end{theorem}
In fact, in \Cref{thm:mcg} we prove the following  stronger fact: any closed subgroup of \(\textrm{MCG}(S)\) containing all the Dehn twists has a unique Polish group topology. 
This formulation includes the subgroup of orientation-preserving mapping classes, which is also commonly referred to as the mapping class group in the case of an orientable surface.

A surface \(S\) is of \emph{finite type} when its fundamental group is finitely generated and of \emph{infinite type} otherwise.
When \( S \) is of finite type, \( \mcg(S) \) is countable and therefore trivially has a unique Polish group structure (coming from the discrete topology). However, when \( S \) is of infinite type, \( \mcg(S) \) is a non-locally compact Polish group.
Therefore, \Cref{thm_1_intro} is mainly a result about \emph{big mapping class groups}, i.e., mapping class groups of infinite-type surfaces. (For a survey, see \cite{AramayonaBig}.)

A Polish group \(G\) is said to have the \emph{automatic continuity property} if every group homomorphism \(G \to H\) is continuous whenever \( H \) is separable.
Note that this is a stronger property than having a unique Polish group topology.
Many Polish groups are known to have the automatic continuity property, including the group \(S_\infty\) of bijections of the natural numbers \cite{KechrisTurbulence}, the group \(\Homeo(2^\mathbb{N})\) of homeomorphisms of the Cantor space \cite{KwiatkowskaGroup}, and the group \(\textrm{Homeo}(M)\) where \(M\) is a compact manifold \cite{MannAutomatic}.
For more on the automatic continuity property, see \cite{RosendalAutomatic}.

A recent result of Bestvina--Domat--Rafi classifies among stable surfaces exactly which surfaces \(S\) have mapping class group with the automatic continuity property.
In particular, there are surfaces \(S\) for which \(\textrm{MCG}(S)\) has the automatic continuity property and surfaces \(S\) for which \(\textrm{MCG}(S)\) does \emph{not} have the automatic continuity property \cite{BestvinaClassification} (see also \cite{DomatBig, MannAutomatica}).
As a corollary, if \(S\) is a stable surface in which every end is telescoping (see \cite{BestvinaClassification} for definitions), then \(\mcg(S)\) has a unique Polish group topology \cite[Corollary~1.6]{BestvinaClassification}. \Cref{thm_1_intro} shows that the mapping class group of \emph{any} surface has a unique Polish group topology, even when the mapping class group fails to have the automatic continuity property.

We also consider mapping class groups of locally finite connected graphs, where for us a graph is a 1-dimensional CW complex.
A graph is \emph{locally finite} if the degree of each vertex is finite.
Recall that the \emph{rank} of a locally finite connected graph is the rank of \( \pi_1(\Gamma) \) as a free group, which is an element of \( \mathbb N \cup \{0,\infty\} \). 
For \(\Gamma\) a locally finite connected graph, \(\mcg(\Gamma)\) is the group of all proper homotopy classes of proper homotopy equivalences \(\Gamma \to \Gamma\).
These groups were defined in \cite{AlgomGroups} as a ``big'' version of the outer automorphism groups of free groups \(\textrm{Out}(\mathbb{F}_n)\).
In fact, for a finite graph, \(\mcg(\Gamma)\) is  \(\textrm{Out}(\mathbb{F}_n)\), where \(n\) is the rank of the graph.
When \(\Gamma\) is infinite, \(\mcg(\Gamma)\) is an uncountable group which supports a non-locally compact Polish group topology \cite{AlgomGroups}.

\begin{theorem}\label{thm_2_intro}
	Let \(\Gamma\) be a locally finite connected graph.
	If \( \mathrm{rank}(\Gamma) \neq 1 \), then \(\mcg(\Gamma)\) has a unique Polish group topology.
\end{theorem}

As in the surface case, we give criteria for a closed subgroup of \( \mcg(\Gamma) \) to have a unique Polish group structure (\Cref{prop: graph subgroup}).
In particular, the \emph{pure mapping class group} \( \pmcg(\Gamma) \) consisting of mapping classes fixing each end of \( \Gamma \) has a unique Polish structure (\Cref{cor: pmcg graph}). 

Our proof of \Cref{thm_2_intro} does not hold for $\mathrm{rank}(\Gamma) = 1$, but we have no reason to believe the theorem is not true in this case. The exact failure is in \Cref{lem:infinite orbit} (see \Cref{rem: failure}).

It is worth noting that other than a few specific cases---for example, when \(\Gamma\) is the full binary tree and hence \(\mcg(\Gamma)\) is the homeomorphism group of Cantor space---it appears it is not known if mapping class groups of locally finite graphs have the automatic continuity property.

For the rank zero case of \Cref{thm_2_intro}, we refer the reader to a result of Kallman \cite{KallmanUniqueness}, which establishes, under mild hypotheses, the uniqueness of the Polish group structure for homeomorphism groups of a topological space.
In particular, if \( \Gamma \) is a rank zero locally finite connected graph, then \( \mcg(\Gamma) \) is isomorphic to the homeomorphism group of its end space \cite{AlgomGroups}.
The end space is compact, second countable, and zero-dimensional, allowing us to apply Kallman's theorem. 

Note that although Kallman's result is broad enough to apply to the homeomorphism group of a surface \(S\), the result does not imply \Cref{thm_1_intro}.
The reason is that although \(\Homeo(S)\) has a unique Polish topology, we do not know---a priori---that every Polish topology on \(\mcg(S)\) is the quotient topology.

If a group \(G\) has a unique Polish topology, then the group \(\Aut(G)\) of automorphisms of \(G\) must preserve this topology, which is to say, each automorphism \(\Phi \colon G \to G\) is continuous.
One consequence of \Cref{thm_1_intro} and \Cref{thm_2_intro} is therefore that automorphisms of mapping class groups of surfaces and locally finite graphs are continuous. 
In the case of big mapping class groups, this is immediate since every automorphism is inner \cite{BavardIsomorphisms}.
In the case of mapping class groups of graphs, these automorphism groups have yet to be determined in generality.
We therefore record this as a corollary, as we believe it will be useful in resolving this question.

\begin{corollary} 
\label{cor:aut contin}
	Let \( \Gamma \) be a locally finite connected graph.
	If \( \mathrm{rank}(\Gamma) \neq 1 \), then every automorphism of \( \mcg(\Gamma) \) is continuous. 
\end{corollary}

Here is how the remainder of the paper is structured.
In Section \ref{sec:algebraic sets}, we prove some general lemmas we need about algebraic sets in topological groups.
In Section \ref{sec:surfaces}, we prove Theorem \ref{thm_1_intro} about surfaces and in Section \ref{sec:graphs} we prove Theorem \ref{thm_2_intro} about locally finite connected graphs.

\subsection*{Acknowledgements}
This project spawned at the BIRS workshop \textit{Blooming Beasts: A Conference on the Topology, Geometry, and Dynamics of Infinite-Type Surfaces}, held in Oaxaca in June 2025.
We would like to thank the organizers for putting on such an enjoyable and productive workshop. We thank George Shaji for helpful conversations.
The first author is supported by the NSERC Discovery Grant RGPIN-2026-04983. The second author is supported by NSF-MSPRF grant DMS-2402039.
The fourth author was supported in part by NSF DMS-2212922 and PSC-CUNY Awards 67380-00 55 and 68352-00 56.

\section{Algebraic sets}
\label{sec:algebraic sets}

	The basic idea behind showing that a Polish group structure on a group is unique is to establish that the identity homomorphism between any two Polish group structures is continuous.
	To accomplish this, we must find enough sets that are Borel in any Polish group structure and appeal to the fact that Borel homomorphisms between Polish groups are automatically continuous \cite[Theorem~9.10]{KechrisClassical}.
	Let us recall some basic definitions so we can set up these required tools.

	Let \( X \) be a topological space.
	The \emph{Borel \(\sigma\)-algebra} on \( X \) is the smallest set containing the open subsets of \( X \) and that is closed under taking complements and countable unions. 
	A subset of \( X \) is \emph{Borel} if it is contained in the Borel \(\sigma\)-algebra.
	Note that the Borel \(\sigma\)-algebra is closed under countable intersections. 
	On occasion, we need to specify the topology with respect to which a given subset is Borel; in this case, given a topology \( \tau \) on a set \( X \), we say that a subset is \emph{\(\tau\)-Borel} if it is Borel in \( (X,\tau) \). 

	A topological space is \emph{Polish} if it is separable and admits a compatible complete metric; a topological group is \emph{Polish} if its underlying topology is Polish.
	A \emph{Polish group structure} on a group \( G \) is a group topology \( \tau \) on \( G \) such that \( (G, \tau) \) is a Polish group. 
	An abstract group homomorphism \( \vp \co G \to H \) between topological groups is \emph{Borel} if \( \vp^{-1}[A] \) is Borel whenever \( A \subseteq H \) is Borel.

	We will use several times the fact that a Borel homomorphism between Polish groups is continuous, as well as the Lusin--Suslin theorem, so we state both here.

	\begin{lemma}[{\cite[Theorem~9.10]{KechrisClassical}}] 
	\label{lem:borel is continuous}
		Every Borel homomorphism between Polish groups is continuous. 	
		\qed
	\end{lemma}

	\begin{theorem}[Lusin--Suslin Theorem {\cite[Theorem~15.1]{KechrisClassical}}]
		If \( f \co X \to Y \) is a Borel injection between Polish spaces, then \( f[U] \) is a Borel subset of \( Y \) for every Borel subset \( U \) of \( X \). 
		\qed
	\end{theorem}

	We will say that a subset \( A \) of a group \( G \) is \emph{algebraic} or \emph{algebraically Borel} if for any Polish group topology \( \tau \) on \( G \), the set \( A \) is \( \tau \)-Borel. The next lemma is key to the rest of the paper. Although it is well-known, we include a proof for completeness.

	\begin{lemma}\label{lem:subbasis}
		Let \( (G, \tau) \) be a Polish group.
		If \( (G, \tau) \) admits a neighborhood subbasis of the identity consisting of algebraically Borel sets, then \( \tau \) is the unique topology on \( G \) making \( G \) a Polish topological group.
	\end{lemma}

	\begin{proof}
		As the property of being algebraic is closed under finite intersections, if \( \tau \) admits a neighborhood subbasis of the identity consisting of algebraic sets, it also admits a neighborhood basis of the identity consisting of algebraic sets.
		Moreover, as \( \tau \) is metrizable, it is first countable; in particular, there exists a countable neighborhood basis, call it \( \mathcal B_1 \), of the identity consisting of algebraic sets.
		Now, as \( \tau \) is separable, there exists a countable dense subset \( \{ g_n \}_{n\in \bn} \) in \( G \).
		It follows that \( \mathcal B = \{ g_n U: n \in \bn, U \in \mathcal B_1\} \) is a countable basis for \( \tau \) consisting of algebraic sets. 
		By the countability of \( \mathcal B \), every open subset is a countable union of algebraic sets, implying that every open set in \( \tau \) is algebraic.
		Therefore, every \( \tau \)-Borel set is algebraic.

		Suppose that \( (G, \mu) \) is a Polish group. 
		We claim \(\mu =\tau\). 
		Since every set in \( \tau \) is algebraic, every set in \( \mathcal B \) is \(\mu\)-Borel, implying that the map \( id \co (G, \mu) \to (G,\tau) \) is a Borel group homomorphism, and hence it is continuous by \Cref{lem:borel is continuous}.
		Therefore, \( \tau \subseteq \mu \). 
		
		Now, as \( id \co ( G, \mu) \to ( G,\tau) \) is injective and Borel, the Lusin--Suslin theorem implies that \( U = id[U] \) is \( \tau \)-Borel for every \( U \in \mu \).
		It follows that the map \( id \co (G, \tau) \to (G,\mu) \) is a Borel group homomorphism and hence continuous, again by \Cref{lem:borel is continuous}.
		Therefore \( \mu \subseteq \tau \), which completes the proof. 
	\end{proof}

	Next, we will need a technique for certifying a set is algebraic.
	In fact, we will need a property stronger than algebraicity.  
	We say a subset \( A \) of a group \( G \) is \emph{algebraically closed} if for any Hausdorff group topology \( \tau \) on \( G \), the set \( A \) is \( \tau \)-closed.
	The key behind our proofs for mapping class groups below is that centralizers of group elements are algebraically closed, which we record in the following lemma.

	\begin{lemma} 
	\label{lem:centralizers}
		The centralizer of an element in a group is an algebraically closed subset of the group.
	\end{lemma}

	\begin{proof}
		Let \( G \) be a Hausdorff topological group.
		The conjugation action of \( G \) on itself is continuous.
		Under the conjugation action, the stabilizer of an element \( g \) in \( G \) is exactly the centralizer of \( g \) in \( G \).
		It follows that the centralizer of any element in any Hausdorff group is closed.
		Therefore, centralizers are algebraically closed subsets in any group.
	\end{proof}

	Our final lemma allows us to take the product of two Borel subgroups to obtain a new Borel subset, under the condition that the intersection of the subgroups is trivial.

	\begin{lemma} 
	\label{lem:product is borel}
		Let \( H \) and \( K \) be Borel subgroups of a Polish group \( G \).
		If \( H \cap K = \{1\} \), then the set \( HK = \{ hk : h \in H, k \in K \} \) is Borel in \( G \).
	\end{lemma}

	\begin{proof}

		As \( H \times K \) is a product of Borel subsets of \( G \), it is Borel in \( G \times G \). 
		Letting \( \vp\co G\times G \to G \) be the multiplication map, we have that \( \vp \) is continuous and \( \vp \restriction_{H\times K} \) is injective, as \( H \cap K = \{1\} \).
		Therefore, \( H K = \vp[H \times K] \) is Borel by the Lusin--Suslin theorem.
	\end{proof}

	We will use \Cref{lem:product is borel} in the case where \( H \) and \( K \) are algebraic; we therefore record this setting as a corollary.

	\begin{corollary}
		\label{cor:product is algebraic}
		Let \( H \) and \( K \) be algebraic subgroups of a group \( G \).
		If \( H \cap K = \{1\} \), then the set \( HK \) is algebraic in \( G \).
		\qed
	\end{corollary}

\section{Surfaces}\label{sec:surfaces}

	The goal of this section is to prove that the mapping class group of a surface has a unique Polish group structure (\Cref{thm:mcg}). 
	This is accomplished by recalling a particular subbasis for the compact-open topology on mapping class groups and a standard fact about centralizers of Dehn twists. 
	After recalling some basics, we record these two facts in the two lemmas below.

	Let \( S \) be a surface (without boundary as we always assume).
	The \emph{mapping class group} of \( S \), denoted \( \mcg(S) \), is the group of isotopy classes of homeomorphisms \( S \to S \). 
	Recall that the \emph{compact-open} topology on \( \Homeo(S) \) has a subbasis given by sets of the form \[ U(K,W) = \{ f \in \Homeo(S) : f(K) \subset W \}, \] where \( K \subseteq S \) is compact and \( W \subseteq S \) is open. 
	Letting \( \Homeo_0(S) \) denote the connected component of the identity in \( \Homeo(S) \), we can identify \( \mcg(S) \) with the quotient group \( \Homeo(S) / \Homeo_0(S) \). 
	We refer to the associated quotient topology on \( \mcg(S) \) as the \emph{compact-open topology} as well.

	A \emph{simple closed curve} in \( S \) refers to the image of an embedding of the circle.
	A simple closed curve is \emph{essential} if no component of its complement has closure homeomorphic to a disk or once-punctured disk; it is \emph{two-sided} if it has a neighborhood homeomorphic to an annulus.

	Given a surface \( S \), let \( \mathcal C_0(S) \) denote the set of isotopy classes of essential two-sided simple closed curves in \( S \).
	Given \( c \in \mathcal C_0(S) \), let \( U_c = \{ f \in \mcg(S) : f(c) = c \} \).
	The Alexander Method implies that \( \{U_c : c \in \mathcal C_0(S) \} \) is a neighborhood basis of the identity in the compact-open topology (the most general version of the Alexander Method for infinite-type surfaces is established in \cite{ShapiroAlexander}).

	\begin{lemma}\label{lem:permutation basis}
		Let \( S \) be an infinite-type  surface.
		The set \[ \{ U_c : c \in \mathcal C_0(S) \} \] is a neighborhood subbasis of the identity in the compact-open topology on \( \mcg(S) \).
		\qed
	\end{lemma}

	Let us recall the most basic of mapping classes, the Dehn twist. 
	Let \( A = \mathbb S^1 \times [0,1] \), where \( \mathbb S^1 \) is the circle.
	Define a homeomorphism \( D \co A \to A \) by \( (\theta, t) \mapsto (\theta + 2\pi t, t) \).
	Note that \( D \) restricts to the identity on \( \partial A \).
	Therefore, given a two-sided simple closed curve \( \alpha \) in \( S \) and an embedding \( \iota \co A \to S \) with \( \iota\left(\mathbb S^1 \times \{1/2\}\right) = \alpha \), we can extend the map \( \iota \circ D \circ \iota^{-1} \) to a homeomorphism \( D_{\alpha, \iota} \co S \to S \) by setting the map to be equal to the identity outside of \( \iota(A) \).
	The isotopy class of \( D_{\alpha,\iota} \) only depends on the isotopy class of \( \alpha \) and the orientation of the image of \( \iota \).
	If we let \( a \in \mathcal C_0(S) \) denote the isotopy class of \( \alpha \), the isotopy class of \( D_{\alpha,\iota} \) is a \emph{Dehn twist about \( a \)} and is denoted \( T_a \in \mcg(S) \).
	If \( S \) is oriented, then we naturally choose the orientation of the image of \( \iota \) to agree with that of \( S \); in this case, we can label \( T_a \) the \emph{left Dehn twist about \( a \)}.

	The following lemma is a direct consequence of \cite[Fact~3.6]{FarbPrimer}.

	\begin{lemma}
		\label{lem:twist conjugates}
		Let \( S \) be a surface.
		If \( a \in \mathcal C_0(S) \) and \( f \in \mcg(S) \), then \( f(a) = a \) if and only if \( fT_af^{-1} \in \{T_a, T_a^{-1}\} \).
		\qed
	\end{lemma}

	\begin{remark}
		\label{rem:twist conjugates}
		The careful reader will observe that \cite[Fact~3.6]{FarbPrimer} appears to imply that \( fT_af^{-1} = T_a \) under the hypotheses of \Cref{lem:twist conjugates}; however, \cite{FarbPrimer} is only considering orientation-preserving mapping classes of orientable surfaces, which accounts for the discrepancy.
		In particular, if \( f \) reverses the orientation of an annular neighborhood of \( a \), then \( f T_a f^{-1} = T_a^{-1} \).
	\end{remark}

	If \( G \) is a closed subgroup of \( \mcg(S) \), we call the associated subspace topology on \( G \) the \emph{compact-open topology}.
	We can now readily establish our main theorem.

	\begin{theorem}
	\label{thm:mcg}
		Let \( S \) be a surface, and let \( G \) be a closed subgroup of \( \mcg(S) \) containing all Dehn twists.
		The compact-open topology on \( G \) is the unique topology on \( G \) making it a Polish group.
	\end{theorem}

	\begin{proof}
		As countable groups have a unique Polish group structure, we may assume that \( S \) is of infinite type.
		Let \( a \in \mathcal C_0(S) \).
		It follows from \Cref{lem:twist conjugates} and \Cref{rem:twist conjugates} that the centralizer of \( T_a \) is either equal to or index two in \( U_a \), the stabilizer of \( a \).
		By \Cref{lem:centralizers}, the centralizer of \( T_a \) is algebraically closed, and hence so is any translate of the centralizer; in particular, \( U_a \) is algebraically closed, as it is either equal to the centralizer or the union of the centralizer and a translate.

		Now, by \Cref{lem:permutation basis}, a subbasis for the identity in \( G \) is given by \( \{ G \cap U_a : a \in \mathcal C_0(S)\}  \), implying \( G \)---with the compact-open topology---admits a neighborhood subbasis for the identity consisting of algebraically closed sets.
		Therefore, by \Cref{lem:subbasis}, if \( (G, \tau) \) is a Polish group, then \( \tau \) must be the compact-open topology.
	\end{proof}


\section{Locally finite graphs}\label{sec:graphs}
	
	Let \( \Gamma \) be a locally finite connected graph.
	Recall that the \emph{rank} of a free group is the cardinality of a minimal generating set. 
	The fundamental group of \( \Gamma \) is a free group of at most countable rank, allowing us to define the \emph{rank} of \( \Gamma \) to be the rank of \( \pi_1(\Gamma) \). 
	We say that \( \Gamma \) is of \emph{finite type} if its rank is finite and it has finitely many ends; otherwise, it is of \emph{infinite type}.

	A map \( f \co X \to Y \) between topological spaces \( X \) and \( Y \) is a \emph{proper homotopy equivalence} if it is proper and there exists a proper map \( g \co Y \to X \) such that \( g\circ f \) and \( f \circ g \) are properly homotopic to the identity. 
	There is a classification of locally finite graphs up to proper homotopy equivalence in terms of their ranks, the homeomorphism type of their ends, and data regarding which ends are limits of loops \cite{AyalaProper}.
	We forgo the full statement here as it will not play a major role in our discussion.

	The \emph{mapping class group} of \( \Gamma \), denoted \( \mcg(\Gamma) \), is the group of proper homotopy classes of proper homotopy equivalences \( \Gamma \to \Gamma \). 
	Observe that if there is a proper homotopy equivalence between two locally finite graphs, then their mapping class groups are isomorphic. 
	
    In the case of a finite graph, that is, of \(\textrm{Out}(\mathbb{F}_n)\), work of Whitehead and Laudenbach shows that there is a closed orientable 3-manifold \(M_\Gamma\) and a surjective homomorphism
    \(\ \mcg(M_\Gamma) \to \mcg(\Gamma) \) with finite kernel; see \cite{BrendleMapping}.
    The following constructions were extended to the infinite-type setting by Udall \cite{UdallSphere}.

	\begin{construction}\label{construction 1}
		For each vertex \( v \) of \( \Gamma \), let \( B_v \) be a closed 3-ball, and for each edge \( e \) in \( \Gamma \), let \( D_e \) be a copy of the \emph{1-handle} \( [0,1] \times D^2 \).
		Fix orientations on each \(B_v\) and \(D_e\). If the edge \( e \) connects vertices \( v \) and \( w \), then identify \( \{ 0 \} \times D^2 \) with a disk in \( \partial B_v \) and identify \( \{ 1\}\times D^2 \) with a disk in \( \partial B_w \) via orientation-reversing homeomorphisms.
		The result is an orientable handlebody, which we denote by \( H_\Gamma \). 
		As the graph is locally finite, the disks chosen for the gluings do not affect the topology of the resulting 3-manifold so long as they are disjoint; therefore, \( H_\Gamma \) is well-defined up to homeomorphism. 
		Now, let \( M_\Gamma \) be the result of doubling \( H_\Gamma \) along its boundary. 
	\end{construction}

	\begin{construction}\label{construction 2}
		The second construction relies on the introduction of building blocks.
		Let \( n, s \in \mathbb N \cup \{0\} \).
		Let \( M_{n,s} \) denote the manifold obtained by taking the connected sum of \( \mathbb S^3 \) with \( n \) copies of \( \mathbb S^2 \times \mathbb S^1 \) and removing \( s \) pairwise-disjoint open balls.  
		For a vertex \( v \) in \( \Gamma \), let \( n_v \) be the number of edges of \( \Gamma \) with both endpoints being \( v \), and let \( s_v \) be the number of edges of \( \Gamma \) that have \( v \) as exactly one of its endpoints.
		Let \( M_v \) be a copy of \( M_{n_v,s_v} \), which we call the \emph{piece} of \( v \) (following \cite{UdallSphere}) and label the boundary components of \( M_v \) by their corresponding edges.
		We now define \( M_\Gamma \) to be \( \bigsqcup M_v /\sim \), where \( \sim \) is generated by identifying boundary components with the same edge labels via orientation-reversing homeomorphisms. 
	\end{construction}

	Both constructions yield, up to homeomorphism, the same 3-manifold \( M_\Gamma \), which is orientable and without boundary.
	Moreover, the homeomorphism type of \( M_\Gamma \) only depends on the proper homotopy type of \( \Gamma \), which we state below.

	\begin{proposition}[{\cite[Proposition~3.5]{UdallSphere}}] 
	\label{prop:classification of 3-manifold}
		Let \( \Gamma \) and \( \Gamma' \) be locally finite connected  graphs.
		There exists a proper homotopy equivalence between \( \Gamma \) and \( \Gamma' \) if and only if \( M_\Gamma \) and \( M_\Gamma' \) are homeomorphic. 
		\qed
	\end{proposition}

	An embedded 2-sphere \( S \) in \( M_\Gamma \) is \emph{essential} if the closure of any component of \( M_\Gamma \ssm S \) is neither homeomorphic to a 3-ball nor a 3-ball with a point deleted from its interior. 
	We will simply refer to essential 2-spheres as \emph{spheres.} 

	The \emph{sphere graph} \( \mathcal S(\Gamma) \) associated to \( \Gamma \) is the graph whose vertex set consists of the isotopy classes of spheres in \( M_\Gamma \) with an edge between two vertices if they admit disjoint representative spheres. 
	Observe that every homeomorphism \( M_\Gamma \to M_\Gamma \) induces a graph automorphism of \( \mathcal S(\Gamma) \), i.e., a bijection on the vertex set of \( \mathcal S(\Gamma) \) that sends adjacent vertices to adjacent vertices and non-adjacent vertices to non-adjacent vertices. 

	\begin{theorem}[{\cite[Theorem~1.1]{HillAutomorphisms}}]
	\label{thm:automorphism}
		Let \( \Gamma \) be an infinite-type locally finite connected  graph.
		\begin{enumerate}[(i)]
			\item Every automorphism of \( \mathcal S(\Gamma) \) is induced by a homeomorphism \( M_\Gamma \to M_\Gamma \).
			\item The group of graph automorphisms of \( \mathcal S(\Gamma) \) is isomorphic to \( \mcg(\Gamma) \). 
				\qed
		\end{enumerate}
	\end{theorem}

	Outside of a few exceptional cases, the theorem above is also true for finite-type graphs, with the finite case considered in \cite{AramayonaAutomorphisms} and the general case in \cite{HillAutomorphisms}. 
	Exactly how one sees the action of \( \mcg(\Gamma) \) on \( \mathcal S(\Gamma) \) is not relevant to our discussion, but we direct the reader to \cite{BrendleMapping,UdallSphere} for details.

	As with surfaces, we define the \emph{mapping class group} of a 3-manifold \( M \), denoted \( \mcg(M) \), to be the group of isotopy classes of homeomorphisms \( M \to M \); if \( M \) has non-empty boundary, we require all homeomorphisms and isotopies to fix \( \partial M \) pointwise.  
	Now, the action of \( \Homeo(M_\Gamma) \) on \( \mathcal S(\Gamma) \) factors through its mapping class group; therefore, \Cref{thm:automorphism} yields an epimorphism \( \mcg(M_\Gamma) \to \autg \). 
	We let \( K_\Gamma \) denote the kernel of this epimorphism. 
	The structure of \( K_\Gamma \) is well understood; in particular, it is an abelian group of exponent two whose nontrivial elements are known as sphere twists (see \cite{BrendleMapping,UdallSphere}).
	With the exception of a brief mention in the proof of \Cref{lem:complements}, this twist structure will not play a serious role in our arguments. 

	\Cref{thm:automorphism} allows us to change our focus from \( \mcg(\Gamma) \) to \( \autg \), whose elements we can view as mapping classes of \( M_\Gamma \). 
	Our goal is now to prove that the group of graph automorphisms of \( \mathcal S(\Gamma) \), denoted \( \Aut(\mathcal S(\Gamma)) \), has a unique Polish topology, namely the permutation topology.
	The \emph{permutation topology} on \( \Aut(\mathcal S(\Gamma)) \) is determined by setting the stabilizers of vertices to be a neighborhood subbasis of the identity.
	As \( \mathcal S(\Gamma) \) is a countable structure, \( \Aut(\mathcal S(\Gamma)) \) equipped with the permutation topology is a Polish group (see \cite[Example 7, p. 59-60]{KechrisClassical}).

	We now proceed to establish that the permutation topology on \( \autg \) is the unique Polish group structure on \( \autg \). 
	To do so, we want to decompose the stabilizer of a vertex into the product of two subgroups, each of which can be realized as the intersection of centralizers, allowing us to appeal to the lemmas in \Cref{sec:algebraic sets}.

	We begin with several definitions.
	Let \( M \) be a \( \pi_1 \)-injective submanifold of \( M_\Gamma \). 
	The \emph{genus} of \( M \), denoted \( g(M) \), is defined to be equal to \( \mathrm{rank}(\pi_1(M)) \) as a free group. If the genus of \( M \) is infinite, set \( g(M) = \infty \).
	We also set \( e(M) \) to be equal to the number of ends of \( M \); if \( M \) has infinitely many ends, set \( e(M) = \infty \).
	The \emph{complexity} of \( M \) is the pair \( (g(M),e(M)) \); the complexity is \emph{finite} if both \( g(M) \) and \( e(M) \) are finite. 
	Given a sphere \( S \) in \( M_\Gamma \), we say that \( S \) is \emph{separating} if \( M_\Gamma \ssm S \) is disconnected; otherwise, it is \emph{non-separating}.
	For a separating sphere \( S \), letting \( M_1 \) and \( M_2 \) denote the two components of \( M_\Gamma \ssm S \), the \emph{complexity} of \( S \) is the pair \( (g(S),e(S)) \in \mathbb N \cup \{0,\infty\} \) defined by \[ (g(S),e(S)) = \min\{(g(M_1), e(M_1)), (g(M_2),e(M_2))\}, \] where we order pairs of integers lexicographically.
	As above, the complexity is \emph{finite} if \( g(S) \) and \( e(S) \) are finite. 	

	Over the next several lemmas, we record several properties of spheres we will need in our arguments. 
	In what follows, given \( s \in \bn\cup\{0\} \), we let \( M_{\Gamma,s} \) denote the 3-manifold obtained by removing \( s \) pairwise-disjoint open balls from \( M_\Gamma \). 

	\begin{lemma} 
	\label{lem:complements}
		Let \( \Gamma \) be a locally finite connected  graph. 
		Suppose \( S \) is a separating sphere in \( M_\Gamma \) and let \( M_1 \) and \( M_2 \) denote the closures of the complementary components of \( S \). 
		\begin{enumerate}[(i)]
			\item For each \( i \), there exists a locally finite connected  graph \( \Gamma_i \) such that \( M_i \) is homeomorphic to \( M_{\Gamma_i,1} \). 
			\item The inclusion \( M_i \hookrightarrow M_\Gamma \) induces a monomorphism \( \mcg(M_i) \to \mcg(M_\Gamma) \).
		\end{enumerate}		
	\end{lemma}

	\begin{proof}
		(i) By construction, there is a finite union of vertex pieces as in \Cref{construction 2}, call it \( N \), containing \( S \).
		Each complementary component \( C \) of \( N \) is therefore also a union of vertex pieces, so has closure homeomorphic to \( M_{\Gamma_C} \) for some locally finite connected  graph \( \Gamma_C \); in fact, \( \Gamma_C \) can be taken to be an induced subgraph of \( \Gamma \).
		Now the closure of each complementary component of \( N \ssm S \) also has this structure, as each such component is homeomorphic to a copy of \( M_{n,s} \) for some \( n,s \in \mathbb N\cup\{0\} \). 
		The desired \( \Gamma_i \) can now be built from these graphs.

		(ii) By extending homeomorphisms of $M_i$ to $M_\Gamma$ by the identity, we have a homomorphism $\theta\co\mcg(M_i) \to \mcg(M_\Gamma)$. 
        Now suppose $\phi \in \mcg(M_i)$ is a non-identity element. 
        If $\phi$ fixes every essential sphere in $M_i$, then $\phi$ is a product of sphere twists about essential spheres. 
        By \cite[Lemma 2.10]{HillAutomorphisms}, every essential sphere in $M_i$ is essential in $M_\Gamma$. 
        Therefore, $\theta(\phi)$ is still a non-trivial product of sphere twists. On the other hand, suppose there is an essential sphere $S_1$ so that $S_2 = \phi(S_1)$ is non-isotopic to $S_1$. Then, after identifying $S_1$ and $S_2$ with their images in $M_\Gamma$, $S_2 = \theta(\phi)(S_1)$ is non-isotopic to $S_1$ by \cite[Proposition 2.9]{HillAutomorphisms}. 
        Therefore $\theta(\phi)$ is not the identity. We conclude that $\theta$ is injective.
	\end{proof}

	For a separating sphere \( S \) in \( M_\Gamma \), we call the graphs \( \Gamma_1 \) and \( \Gamma_2 \) given by \Cref{lem:complements}(i) the \emph{graph components of \( S \)}.
	Note that the graph components of \( S \) are well-defined up to proper homotopy equivalence. 

	\begin{lemma} 
	\label{lem:change of coordinates}
	Let \( \Gamma \) be a locally finite connected  graph, and let \( s \in \mathbb N \cup \{0\} \).
		\begin{enumerate}[(i)]
			\item If \( S_1 \) and \( S_2 \) are non-separating spheres in \( M_{\Gamma,s} \), then there exists a homeomorphism \( f \co M_{\Gamma,s} \to M_{\Gamma,s} \) such that \( f(S_1) = S_2 \). 
				Moreover, \( f \) can be taken to fix each end of \( M \). 

			\item Let \( S_1 \) and \( S_2 \) be separating spheres in \( M_{\Gamma,s} \).
				Let \( \Gamma_i \) and \( \Gamma_i' \) be the graph components of \( S_i \). 
				If, up to relabeling, \( \Gamma_1 \) is proper homotopy equivalent to \( \Gamma_2 \) and \( \Gamma_1' \) is proper homotopy equivalent to \( \Gamma _2' \), then there exists a homeomorphism \( f \co M_{\Gamma,s} \to M_{\Gamma,s} \) such that \( f(S_1) = S_2 \). 
		\end{enumerate}
	\end{lemma}

	\begin{proof}
		We first consider (ii).
		Let \( M_i \) and \( M_i' \) be the closures of the complementary components of \( S_i \).
		By definition, \( M_i \) (resp., \( M_i' \)) is homeomorphic to \( M_{\Gamma_i,1} \) (resp., \( M_{\Gamma_i',1} \)). 
		Under the assumptions, \Cref{prop:classification of 3-manifold} implies that there exist homeomorphisms \( f_1 \co M_1 \to M_2 \) and \( f_2 \co M_1' \to M_2' \). 
		Now, it follows from the Alexander trick that \( \mathrm{Homeo(\mathbb S^2)} \) has two connected components, which allows us to glue \( f_1 \) and \( f_2 \) together to construct \( f \co M_\Gamma \to M_\Gamma \), which satisfies \( f(S_1) = S_2 \).  

		Let us turn to (i). 
		Arguing as in \Cref{lem:complements}, the closure of the complementary component of \( S_i \) can be written \( M_{\Gamma_i,2} \) for some locally finite connected  graph \( \Gamma_i \). 
		It is clear that \( \Gamma_1 \) and \( \Gamma_2 \) have the same rank and homeomorphic end spaces, and therefore they are proper homotopy equivalent. 
		One can now follow the argument above for the separating case.
	\end{proof}

	With these basic lemmas at hand, we move to establishing a desirable neighborhood subbasis for the identity in the permutation topology. 
	We require sets in the subbasis to be expressible in terms of centralizers, allowing us to access the tools from \Cref{sec:algebraic sets}. 
	To do so, we will need to restrict ourselves to a subbasis consisting of stabilizers of separating spheres that have high enough complexity.
	In particular, for a separating sphere \( S \), we will require that \( (g(S), e(S)) > (1,1) \). 
	As will be exhibited in the arguments below, this requirement comes about from needing every sphere contained in a complementary component \( M \) of \( M_\Gamma \ssm S \) to have infinite orbit under \( \mcg(M) \) (viewed as a subgroup of \( \mcg(M_\Gamma) \)).
    
    We say a separating sphere is \emph{basic} if \( (g(S),e(S)) > (1,1) \).
	In the next lemma, we record the fact that spheres in a complementary component of a basic separating sphere have infinite orbit in the stabilizer of the basic sphere. 
	This can be deduced using sphere pushes, or their graph equivalent word maps (introduced in \cite[Section~3.3]{DomatCoarse}).

	\begin{lemma}\label{lem:infinite orbit}
		Let \( \Gamma \) be a locally finite connected  graph.
		If \( (\mathrm{rank}(\Gamma), e(\Gamma)) > (1,1) \), then every sphere in \( M_\Gamma \) has infinite \( \mcg(M_\Gamma) \) orbit.
		\qed
	\end{lemma}

	\begin{remark}\label{rem: failure}
		If \( \mathrm{rank}(\Gamma) = 0 \), then it is possible for \( M_\Gamma \) to have spheres that are fixed by every element of \( \mcg(M_\Gamma) \).
		For example, if the end space of \( \Gamma \) can be decomposed as the disjoint union of \( C \) and \( F \), where \( C \) is a Cantor set and \( F \) is a finite set of cardinality at least two, then there is a unique sphere in \( M_\Gamma \) that separates \( C \) and \( F \) (where we are identifying the end space of \( \Gamma \) and \( M_\Gamma \)), and hence it must be fixed by every element of \( \mcg(M_\Gamma) \). 
		This is an issue for rank one graphs, because any separating sphere \( S \) must cutoff a genus zero submanifold, creating the possibility that spheres in this submanifold fail to have infinite orbit in the stabilizer of \( S \).
		This is the reason we must require the rank of \( \Gamma \) to be at least two for our arguments. 
	\end{remark}

	Let \( S \) be a  sphere in \( M_\Gamma \).
	Viewing \( S \) as a vertex in \( \mathcal S(\Gamma) \), let \( U_S \) denote its stabilizer in \( \autg \).
	The goal of the following lemma, \Cref{lem:new subbasis}, is to establish that the set 
	\[ \{ U_S: S \text{ a basic separating sphere} \} \]
	is a neighborhood subbasis of the identity. 
		
	\begin{lemma} 
	\label{lem:new subbasis}
		Let \( \Gamma \) be an infinite-type locally finite connected  graph of rank at least two.
		If \( S \) is either a non-separating sphere or a non-basic separating sphere, then there exist basic separating spheres \( S_1 \) and \( S_2 \) such that \( U_{S_1} \cap U_{S_2} \subseteq U_S \). 
	\end{lemma}

	\begin{proof}
		\begin{figure}
			\centering
			\includegraphics[scale=0.9]{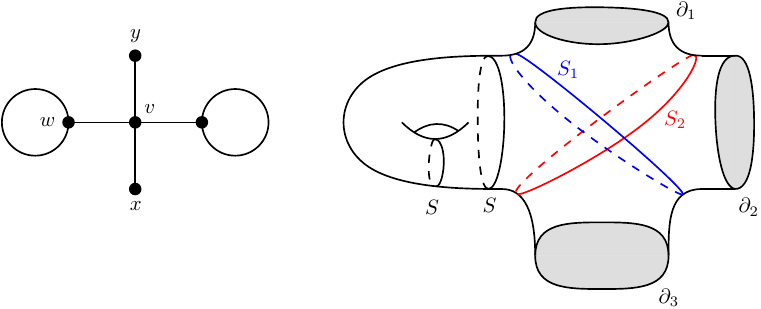}
			\caption{A neighborhood of the vertices \( v \) and \( w \) (left) and the associated handlebody \( H \) (right) from the proof of \Cref{lem:new subbasis}.
            The non-separating or non-basic separating spheres $S$ correspond to the loop and separating edges incident to $w$, respectively. The basic spheres $S_1$ and $S_2$ are pictured, but do not correspond to edges in the graph.}
			\label{fig:non-separating}
		\end{figure}

		Let us first consider the case in which \( S \) is either non-separating or has a complementary component homeomorphic to the interior of \( M_{1,1} \). 
    Up to proper homotopy equivalence of \( \Gamma \), this \( M_{1,1} \) is the vertex piece associated to a vertex \( w \in \Gamma \). 
    We may also assume that the unique vertex adjacent to \( w \) is \( v \), whose vertex piece is the interior of \( M_{0,4} \).
    The union of the pieces of \( v \) and \( w \) is the double of a handlebody \( H \) in \Cref{fig:non-separating}.

		In the piece of \( v \), which is homeomorphic to \( M_{0,4} \), there are exactly three spheres, call them \( S_1 \), \( S_2 \), and \( S_3 \) (two of these spheres are pictured in \Cref{fig:non-separating}). 
		Note that we can edit \( \Gamma \) with a proper homotopy equivalence to change the topology of the complementary components of the piece of \( v \).
		We claim that there is such an edit in which \( S_1 \) and \( S_2 \) are basic.

		We prove this by adding restrictions on the other vertices. 
		First, we may assume that the piece of the unlabeled vertex incident to \(v\) is also \( M_{1,1} \), since the rank of \( \Gamma \) is at least two. 
		There are two other vertices incident to \( v \), call them \( x \) and \( y \). 
		Because \( \Gamma \) is of infinite type, we may assume that \( x \) is a cut vertex whose complement not containing \(v\) is not compact. 
		If \( \Gamma \) has one end, then it has infinite rank, whence we may assume that the piece of \( y \) is also \( M_{1,1} \).
		Otherwise we may assume that \( y \) is also a cut vertex whose complement not containing \(v\) is not compact.
		From this it follows by inspection that \( S_1 \) and \( S_2 \) are basic.
		
		Observe that any element of \( U_{S_1} \cap U_{S_2} \) must preserve the piece of \( v \) (up to isotopy), as it is a regular neighborhood of \( S_1 \cup S_2 \).
		Moreover, once \( S_1 \) and \( S_2 \) are stabilized, \( S_3 \) must also be stabilized, as it is the only remaining sphere in the piece of \( v \). 
		It follows that the boundary components of the piece of \( v \) cannot be permuted; in particular, it must fix the piece of \( w \), and hence \( S \), since \( S \) is the unique sphere in \( M_\Gamma \) contained in the piece of \( w \) which is separating or non-separating, respectively.
		This establishes \( U_{S_1} \cap U_{S_2} \subset U_S \).

	\begin{figure}
		\centering
		\includegraphics[scale=0.9]{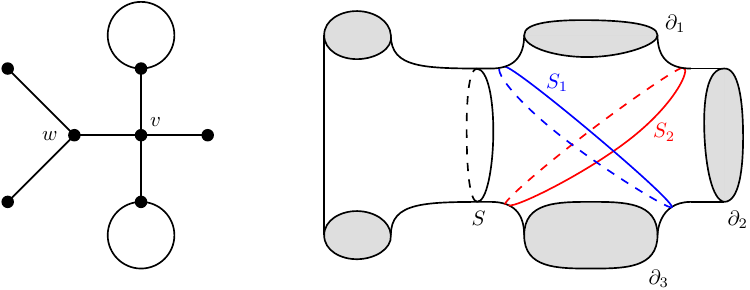}
		\caption{A neighborhood of the vertices \( v \) and \( w \) (left) and the associated handlebody \( H \) (right) from the proof of \Cref{lem:new subbasis}. In this case the non-basic separating sphere $S$ corresponds to the edge between vertices $v$ and $w$. The basic separating spheres $S_1$ and $S_2$ are pictured, but do not correspond to edges in the graph.}
		\label{fig:separating}
	\end{figure}

		For the case where \( S \) cuts off a genus zero submanifold, an analogous argument to the \( M_{1,1} \) case applies with \Cref{fig:non-separating} replaced by \Cref{fig:separating}.
		We leave the details to the reader.
	\end{proof}

	Given a homeomorphism \( f \co M \to M \), the \emph{support} of \( f \), denoted \( \supp(f) \), is the closure in \( M \) of the set \( \{ x \in M : f(x) \neq x \} \). 
	Given a subset \( X \) of \( M \) and a mapping class \( \vp \in \mcg(M) \), we say that \( \vp \) is supported in \( X \) if there exists a representative \( f \) of \( \vp \) whose support is contained in \( X \). 
	Similarly, we say \( \vp \in \autg \) is supported in \( X \) if there is a homeomorphism \( f \co M_\Gamma \to M_\Gamma \) inducing \( \vp \) on \( \mathcal S(\Gamma) \) and whose support is contained in \( X \).

	\begin{lemma}\label{lem:support of centralizers}
		Let \( \Gamma \) be a locally finite connected  graph. 
		Suppose \( S \) is a basic separating sphere in \( M_\Gamma \), and let \( M \) be the closure of a component of \( M_\Gamma \ssm S \).
		Let \( U \) consist of the elements of \( \autg \) supported in \( M \).
		If \( \varphi \in \autg \) commutes with each element of \( U \), then \( \varphi(S) = S \).
		Moreover, there exists \( f \co M_\Gamma \to M_\Gamma \) representing \( \varphi \) fixing \( S \) pointwise and satisfying \( f[M] = M \). 
	\end{lemma}

	\begin{proof}
		For the sake of arguing by contradiction, suppose that \( \varphi(S) \neq S \). 
		We consider two cases, depending on whether \( \varphi(S) \) and \( S \) are adjacent in the sphere complex or not. 

		First, suppose that \( \varphi(S) \) and \( S \) are adjacent, so up to isotopy, they are disjoint; in particular, \( \varphi(S) \) is either contained in \( M \) or \( M_\Gamma \ssm M \).
		In this case, \( \varphi \) must map a sphere \( S' \) contained in \( M \) into the complement of \( M \), or vice versa. 
		Start with the case where \( S' \) is in \( M \).
		By \Cref{lem:infinite orbit}, we can choose \( \psi \in U \) such that \( \psi(S') \neq S' \). 
		As \( \psi \in U \), it must be that \( \psi(\varphi(S')) = \varphi(S') \).
		But this yields
			\[
				\varphi(S') = \psi(\varphi(S')) = \varphi(\psi(S')),
			\]
		implying \( S' = \psi(S') \), a contradiction. 
		Repeating the argument with \( \varphi \) replaced by \( \varphi^{-1} \) handles the other case.

		Next, suppose that \( \varphi(S) \) and \( S \) are not adjacent, so that any two representatives intersect. 
		Let \( \psi \in U \). 
		Now, \( \psi(S) = S \) and \( \psi \) commutes with \( \varphi \), implying \( \psi \) fixes \( \varphi(S) \); indeed, 
		\( \psi(\varphi(S)) = \varphi(\psi(S)) = \varphi(S). \)
		Therefore, every element of \( U \) fixes \( \varphi(S) \).

		We may choose representatives of \( S \) and \( \varphi(S) \) in \( M_\Gamma \) such that the number of components in \( S \cap \varphi(S) \) is minimal over all representatives.
		Then the intersection of \( \varphi(S) \) and \( M \) is a union of planar surfaces \( \Sigma_1, \ldots, \Sigma_k \) such that each component of \( \partial \Sigma_i \) is a circle in \( S \). 
		Because \( S \) is separating, each \( \Sigma_i \) must have at least two boundary components, and therefore there is a unique collection of pairwise disjoint disks in \( S \) whose union with \( \Sigma_i \) is a sphere, which can be isotoped into \( M \).  
		Let \( P \) denote the collection of all such spheres (forgetting any non-essential spheres).  
		Then, as \( \varphi(S) \) has nontrivial intersection with \( S \), \( P \) is a non-empty finite set of spheres canonically associated to the pair \( S \) and \( \varphi(S) \).
		It follows that every element of \( U \) preserves the set of spheres \( P \).
		But this contradicts \Cref{lem:infinite orbit}, allowing us to conclude that \( \varphi(S) = S \).

		To finish, observe that our initial argument showing that \( \varphi(S) \) and \( S \) cannot be distinct and adjacent also implies that \( \varphi \) cannot move a sphere in \( M \) to its complement.
		It follows that the desired representative \( f \) exists. 
	\end{proof}

	To continue, we need a detailed understanding of non-separating spheres in \( M_{1,2} \), which is provided by the following lemma (see also \cite[Section~2.3]{UdallSphere}).	

	\begin{lemma}\label{lem: m12}
		Let \( S \) be a non-separating sphere in \( M_{1,2} \). 
		\begin{enumerate}[(1)]
			\item There are exactly three spheres disjoint from \( S \), one of which is separating.
			\item Let \( f \in \mcg(M_{1,2}) \) be a nontrivial mapping class such that \( f(S) = S \) and such that \( f \) stabilizes each component of \( \partial M_{1,2} \). 
				Then \( f \) transposes the non-separating spheres disjoint from \( S \) and \( f(S') \neq S' \) for every non-separating sphere \( S' \neq S \). 
				Moreover, if \( g \in \mcg(M_{1,2}) \) satisfies \( g(S) = S \), then either (i) \( g(S') = S' \) for every sphere \( S' \) or (ii) \( g(S') = f(S') \) for every sphere \( S' \). 
		\end{enumerate}
	\end{lemma}

	\begin{proof}
		Let us start with \( M_{0,4} \), and label its boundary components \( \partial_1, \partial_2, \partial_3 \), and \( \partial_4 \). 
		By \cite[Lemma~9]{BeringFinite}, \( M_{0,4} \) has exactly three spheres, each separating and each corresponding to a partition of \( \partial M_{0,4} \).
		Now, observe that if we cut \( M_{1,2} \) along \( S \), we obtain a copy of \( M_{0,4} \), implying there are three spheres disjoint from \( S \). 
		A sphere in \( M_{0,4} \) becomes non-separating in \( M_{1,2} \) exactly when it separates \( \partial_1 \) from \( \partial_2 \). 
		It follows that only one of the three spheres disjoint from \( S \) in \( M_{1,2} \) is separating. 

	  Continuing to view \( M_{0,4} \) as \( M_{1,2} \) cut along \( S \), let \( f \in \mcg(M_{0,4}) \) be the mapping class transposing \( \partial_1 \) and \( \partial_2 \) while fixing both \( \partial_3 \) and \( \partial_4 \). 
		Then \( f \) induces a map on \( M_{1,2} \) fixing \( S \) and transposing the two non-separating spheres disjoint from \( S \). 
		Abusing notation, let us view \( f \in \mcg(M_{1,2}) \).
		Now, if \( g \in \mcg(M_{1,2}) \) also fixes \( S \) and each boundary component of \( \partial M_{1,2} \), then by restricting \( g \), we can view \( g \in \mcg(M_{0,4}) \).
		In \( M_{0,4} \), \( g \) either permutes \( \partial_1 \) and \( \partial_2 \) or fixes them.
		If it fixes them, then it fixes every sphere in \( M_{0,4} \), implying it fixes every sphere in \( M_{1,2} \).
		If it permutes them, then \( g^{-1}\circ f \) fixes every sphere in \( M_{0,4} \), implying \( g(S') = f(S') \) for every sphere \( S' \) in  \( M_{1,2} \). 

		Let us finish by arguing that \( S \) is the unique non-separating sphere fixed by \( f \). 
		Consider the graph whose vertex set consists of the non-separating spheres in \( M_{1,2} \), and where the edge relation corresponds to admitting disjoint representatives. 
		By (1), we can see this graph is just the real line with vertices at every integer point. 
		It is clear \(\mcg( M_{1,2}) \) acts by automorphisms on this graph. 
		The map \( f \)  induces the reflection through the vertex corresponding to \( S \), implying \( S \) is the unique fixed non-separating sphere.  
	\end{proof}

	\begin{lemma} 
	\label{lem:intersection of centralizers}
		Let \( \Gamma \) be a locally finite connected graph. 
		Suppose \( S \) is a basic separating sphere in \( M_\Gamma \), and let \( M_1 \) and \( M_2 \) denote the closures of the components of \( M_\Gamma \ssm S \).
		If \( U_i \) consists of the elements of \( \autg \) supported in \( M_i \) and \( C_i = C(U_i) \) consists of the elements of \( \autg \) that commute with all the elements of \( U_i \), then \( C(U_1) = U_2 \). 
	\end{lemma}

	\begin{proof}
		First, observe that \( U_2 \subset C_1 \). 
		So, we need only concern ourselves with one containment; to this end, fix \( \vp \in C_1 \). 	
		By \Cref{lem:support of centralizers}, \( \varphi(S) = S \); moreover, we can choose a homeomorphism \( f \co M_\Gamma \to M_\Gamma \) that induces \( \vp \) on \( \mathcal S(\Gamma) \) and such that \( f \) restricts to the identity on \( S \) and \( f[M_1] = M_1 \).
		We can therefore write \( f = f_1 \circ f_2 \) with \( f_i \) supported in \( M_i \).
		Let \( \varphi_i \) denote the isotopy class of \( f_i \), so that \( \varphi_i \in U_i \) and \( \varphi = \varphi_1\varphi_2 \). 
		We claim \( \varphi_1 \) is the identity. 

		Fix a non-separating sphere \( \Sigma \) in \( M_1 \); note that such a sphere exists as \( S \) is basic.  
		Again using that \( S \) is basic, we can choose a copy of \( M_{1,2} \), denoted \( N \), in \( M_1 \) containing \( \Sigma \) and having \( S \) as a boundary component. 

		By \Cref{lem:support of centralizers}, every element of \( U_1 \) fixes \( S \).
		Therefore, by \Cref{lem: m12}, there is a unique nontrivial element of \( U_1 \) supported in \( N \) fixing \( \Sigma \), call it \( \psi \). 
		(In \( N \), there are exactly two non-separating spheres disjoint from \( \Sigma \), and \( \psi \) swaps them.)
		It follows that \( f_1[N] \) is isotopic to \( N \), as \( \vp_1 \) commutes with \( \psi \) and \( N \) is the support of \( \psi \) (up to isotopy). 
		Moreover, as \( \vp_1(S) = S \), \( \vp_1 \) fixes each of the isotopy classes of the boundary components of \( N \), allowing us to write \( \vp_1 = \eta_1\eta_2 \) with \( \eta_1 \) supported in \( N \) and \( \eta_2 \) supported in the closure of \( M_1 \ssm N \). 

		Now, \( \Sigma \) is the unique non-separating sphere in \( N \) fixed by \( \psi \).
		Hence, \( \eta_1 \) commuting with \( \psi \) implies that \( \eta_1(\Sigma) = \Sigma \).
		In particular, \[ \vp_1(\Sigma) = \eta_1(\Sigma) = \Sigma. \]
		Therefore, as \( \Sigma \) was arbitrary, \( \vp_1 \) fixes every non-separating sphere in \( M_1 \). 

		We now must show that \( \vp_1 \) fixes every separating sphere.
		Let \( R \) be a separating sphere in \( M_1 \), and let \( \Sigma_0 \) be a non-separating sphere in \( M_1 \) disjoint from \( R \). 
		First assume that \( R \) does not bound a copy of \( M_{1,1} \). 
		Then we can choose a copy of \( M_{1,2} \), call it \( N \), containing \( \Sigma_0 \) and having \( R \) as a boundary component. 
		Let \( R' \) denote the other boundary component of \( N \). 
		In \( N \), there are two non-separating spheres \( \Sigma_1 \) and \( \Sigma_2 \) disjoint from \( \Sigma_0 \). 
		A regular neighborhood of \( \Sigma_1 \cup \Sigma_2 \) is a copy of \( M_{0,4} \) with a boundary component isotopic to \( R \), a boundary component isotopic to \( R' \), and the remaining boundary components each isotopic to \( \Sigma_0 \). 
		In particular, as \( \vp_1 \) fixes each of the \( \Sigma_i \), we can conclude that \( \vp_1 \) fixes \( R \) as well.

		Now, assume that \( R \) bounds a copy of \( M_{1,1} \).
		We can then find a copy of \( M_{1,2} \), call it \( N \), containing both \( R \) and \( \Sigma_0 \).
		The above argument shows that \( \vp_1 \) fixes \( \Sigma_0 \) and each of the boundary components of \( N \). 
		It follows that we can write \( \vp_1 = \eta_1\eta_2 \) with \( \eta_1 \) supported in \( N \) and \( \eta_2 \) supported in the closure of \( M_1 \ssm N \). 
		Now, \( R \) is the unique separating sphere in \( N \) disjoint from \( \Sigma_0 \).
		Therefore, as \( \eta_1 \) fixes \( \Sigma_0 \), it must also fix \( R \), and hence \( \vp_1(R) = \eta_1(R) = R \). 

		We have shown that \( \vp_1 \) fixes each sphere contained in \( M_1 \).
		Therefore, as an element of \( \mcg(M_1) \) it is either trivial or a product of sphere twists.
		By \Cref{lem:complements}, the same is true as an element of \( \mcg(M_\Gamma) \).
		In either case, it is trivial as an element of \( \autg \).
		Therefore, \( \vp = \vp_2 \in U_2 \), implying \( U_2=C_1 \). 
 	\end{proof}

	\begin{theorem} 
	\label{thm:graphs}
		If \( \Gamma \) is a locally finite connected  graph of rank at least two, then \( \mcg(\Gamma) \) has a unique Polish group structure.
	\end{theorem}

	\begin{proof}
		Every countable group has a unique Polish group structure, namely the one given by the discrete topology.
		We may therefore assume that \( \Gamma \) is of infinite type.

		By \Cref{thm:automorphism}, it suffices to prove that the permutation topology on \( \autg \) is the unique Polish group structure on \( \autg \). 
		By \Cref{lem:subbasis} and the definition of the permutation topology, it suffices to show that stabilizers of vertices are algebraic in \( \autg \). 
		By \Cref{lem:new subbasis}, it suffices to show that \( U_S \) is algebraic, where \( S \) is a basic separating sphere.

		Let \( S \) be a basic separating sphere.
		Let \( M_1 \) and \( M_2 \) denote the closures of the components of \( M_\Gamma \ssm S \). 
		Let \( U_i \) denote the elements of \( \autg \) that can be represented by a homeomorphism of \( M_\Gamma \) supported in \( M_i \). 
		For \( g \in \autg \), let \( C(g) = \{ h \in \autg: hg = gh \} \) denote the centralizer of \( g \). 
		By \Cref{lem:intersection of centralizers}, 
		\[ U_i = \bigcap_{g \in U_{3-i}} C(g). \]
		By \Cref{lem:centralizers}, \( C(g) \) is algebraically closed, and hence \( U_i \) is algebraically closed, as an arbitrary intersection of algebraically closed sets is algebraically closed. 
		By definition, \( U_1 \cap U_2 = \{1\} \), implying \( U_1U_2 \) is algebraic by \Cref{cor:product is algebraic}.
		
		Now, \( U_1U_2 \) is precisely the subgroup of \( \autg \) consisting of elements admitting a representative homeomorphism of \( M_\Gamma \) that fix each of \( M_1 \) and \( M_2 \) setwise.
		If there is no homeomorphism in \( U_S \) swapping \( M_1 \) and \( M_2 \), then every element of \( U_S \) can be written as \( g_1g_2 \) with \( g_i \in U_i \), implying \( U_S = U_1U_2 \) is algebraic.
		If there is a homeomorphism \( \sigma \) of \( M \) swapping \( M_1 \) and \( M_2 \), then \( U_S = U_1U_2 \cup \sigma U_1U_2 \). 
		As \( U_1U_2 \) is algebraic, so is \( \sigma U_1U_2 \), implying their union---and hence \( U_S \)---is algebraic. 
	\end{proof}

	\subsection{Pure mapping class group}

	Unlike in the surface case, it not as simple to give a condition that guarantees a closed subgroup of \( \mcg(\Gamma) \) has unique Polish topology.
	However, we describe here a set of conditions to be able to directly use the arguments above, which we can apply to  pure mapping class groups.

	\begin{proposition}\label{prop: graph subgroup}
		Let \( \Gamma \) be a locally finite connected graph of rank at least two, and let \( G \) be a closed subgroup of \( \mcg(\Gamma) \). 
		Suppose that 
		\begin{enumerate}[(i)]
			\item \( G \) contains all word maps, and
			\item given a submanifold \( M \) of \( M_\Gamma \) homeomorphic to \( M_{1,2} \) with both boundary components being separating spheres and given two non-separating spheres \( S_1 \) and \( S_2 \) in \( M \), there exists an element of \( G \) with a representative in \( \mcg(M_\Gamma) \)  supported in \( M \) and mapping \( S_1 \) to \( S_2 \). 
		\end{enumerate}
		Then \( G \) has a unique Polish group structure.
		\qed
	\end{proposition}

	\Cref{prop: graph subgroup} is verified simply by following the arguments above line-by-line and simply replacing any open set with its intersection with \( G \). 

	The \emph{pure mapping class group} of a locally finite connected graph \( \Gamma \), denoted \( \pmcg(\Gamma) \) is the subgroup of \( \mcg(\Gamma) \) consisting of mapping classes that fix each end of \( \Gamma \).  
	We similarly define the the pure mapping class group of a manifold \( M \), denoting it by \( \pmcg(M) \). 
	In the epimorphism \( \mcg(M_\Gamma) \to \mcg(\Gamma) \), the image of \( \pmcg(M_\Gamma) \) is \( \pmcg(\Gamma) \). 
	Note that (1) pure mapping class groups are closed, and (2) the pure mapping class group of a rank zero graph is trivial.

	\begin{corollary}\label{cor: pmcg graph}
		If \( \Gamma \) is a locally finite connected graph of rank at least two, then \( \pmcg(\Gamma) \) has a unique Polish group structure.
		\qed
	\end{corollary}

\bibliographystyle{amsplain}
\bibliography{references}

\end{document}